\documentclass[11pt,twoside]{article}

\usepackage[left=3cm,right=3cm,top=3cm,bottom=3cm]{geometry}
\usepackage{url}
\usepackage{amsmath,amssymb,amsthm,mathtools}
\numberwithin{equation}{section}

\usepackage{graphicx,epsfig}
\usepackage{xcolor}

\usepackage{enumerate}
\usepackage{multicol}
\usepackage{leftindex,tensor,mhchem}
\usepackage{scalerel,stackengine}
\stackMath

\newcommand\reallywidehat[1]{%
	\savestack{\tmpbox}{\stretchto{%
			\scaleto{%
				\scalerel*[\widthof{\ensuremath{#1}}]{\kern-.6pt\bigwedge\kern-.6pt}%
				{\rule[-\textheight/2]{1ex}{\textheight}}%
			}{\textheight}}{0.5ex}}%
	\stackon[1pt]{#1}{\tmpbox}%
}

\theoremstyle{plain}
\newtheorem{theorem}{Theorem}

\newtheorem{proposition}{Proposition}

\theoremstyle{definition}
\newtheorem{definition}{Definition}

\newtheorem{example}{Example}

\theoremstyle{remark}
\newtheorem{remark}{Remark}

\usepackage{eso-pic}
\makeatletter
\AddToShipoutPicture{
	\setlength{\@tempdimb}{.5\paperwidth}
	\setlength{\@tempdimc}{.5\paperheight}
	\setlength{\unitlength}{1pt}
	\put(\strip@pt\@tempdimb,\strip@pt\@tempdimc){
	}
}
\makeatother

\begin{document}
	\setcounter{page}{1}
	\thispagestyle{empty}
	
	\pagestyle{myheadings}
	\markboth{Chaudhary and Devi }{Chaudhary and Devi}
	
	\vspace{0.1in}

		{\baselineskip 20truept
			\begin{center}
				{\Large \bf Record Values Based Inaccuracy Measures with an Application to Testing Symmetry}
				\footnote{\noindent
					{\bf 1.} Corresponding author  
					E-mail: skchaudhary1994@gmail.com, santosh.chaudhary@cuj.ac.in}
		\end{center}}
		
		\vspace{0.1in}
		
		\begin{center}
			{\large \bf Santosh Kumar Chaudhary$^{1}$ and Bhagwati Devi$^{2}$}\\[0.1in]
			{\large \it
				$^{1, 2}$ Department of Statistics, Central University of Jharkhand,  
				Cheri-Manatu, Ranchi – 835222, Jharkhand, India}
		\end{center}
		
		\vspace{0.1in}
		\baselineskip 12truept
		
		\begin{abstract}
			This paper studies information-theoretic inaccuracy measures for record values
			and their application to testing symmetry.
			We study the Kerridge inaccuracy measure, cumulative residual inaccuracy,
			cumulative past inaccuracy, and extropy-based inaccuracy measures to the
			distributions of $n$th upper and lower $k$-record values.
			Explicit expressions are derived for several common lifetime distributions
			(exponential, Pareto, Weibull, and uniform), and the monotonic behaviour of
			these measures with respect to the record order $n$, the parameter $k$, and the
			distributional parameters is analysed.
			Building on the known characterisation that equality of upper and lower
			record-based inaccuracy implies symmetry, we develop a nonparametric
			goodness-of-fit test for symmetry.
			The test statistic is the difference between the estimated Kerridge inaccuracy
			measures for upper and lower records, and its null distribution is obtained by a
			bootstrap procedure that enforces symmetry.
			A simulation study demonstrates that the test maintains its nominal level for a
			variety of symmetric distributions and achieves high power against skewed
			alternatives, while a real-data illustration on annual maximum temperatures
			confirms its practical usefulness.
			
			\vspace{0.1in}
			
			
		\end{abstract}

	\section{Introduction}
	
	Inaccuracy measures play a fundamental role in information theory and statistics by quantifying the discrepancy between probability distributions. Such measures provide insight into the amount of information lost or gained when one distribution is used to approximate another. Among the earliest and most influential measures is Shannon entropy, which quantifies uncertainty and has motivated the development of several related measures, including Kerridge inaccuracy. These measures have found wide-ranging applications in coding theory, statistical inference, reliability analysis, and lifetime modeling. 
	
	In recent years, growing attention has been directed toward extending inaccuracy measures to functions of order statistics and record values. Record values arise naturally in many practical situations, such as meteorology, reliability testing, and environmental studies. The present study focuses on inaccuracy measures associated with record values, including Kerridge inaccuracy, extropy-based inaccuracy, and their cumulative variants. Inaccuracy measures based on record values have received little attention in the literature, and their properties remain largely unexplored. This paper aims to fill this gap by conducting a systematic study of several information‑theoretic inaccuracy measures for the $n$th upper and lower $k$-record values.
	
	\subsection{Inaccuracy Measures}
	
	Inaccuracy measures, including Kerridge’s measure and Shannon entropy, assess the adequacy of statistical models in representing observed data. These measures have been extensively studied in coding theory (Nath (1968)). Kumar and Taneja (2015) introduced the cumulative residual inaccuracy measure, while Tahmasebi et al. (2018) proposed a cumulative inaccuracy measure for lower record values and established characterization results in the dynamic case. Thapliyal and Taneja (2013) examined inaccuracy measures between the distributions of order statistics and the parent random variable. More recently, Mohammadi and Hashempour (2024) introduced the dynamic cumulative residual extropy inaccuracy (DCREI) measure and its weighted version by extending the dynamic cumulative residual extropy framework. Tahmasebi and Daneshi (2018) considered a measure of inaccuracy between distributions of the nth record value and parent random variable.

	Let $X$ and $Y$ be two non-negative random variables with cumulative distribution functions (cdf) $F_X(x)$ and $F_Y(x)$, probability density functions (pdf) $f_X(x)$ and $f_Y(x)$, and reliability functions $\bar{F}_X(x)$ and $\bar{F}_Y(x)$, respectively. Kerridge (1961) introduced a measure of inaccuracy, known as Kerridge inaccuracy, as a generalization of Shannon entropy. If $F_X(x)$ represents the true distribution corresponding to observations and $F_Y(x)$ denotes the distribution assigned by the experimenter, the Kerridge (1961) inaccuracy measure is defined as
	\begin{align}\label{1}
		H(f_X,f_Y) = - \int f_X(x)\log f_Y(x)\,dx.
	\end{align}
	It is evident that $H(f_X,f_X)=H(X)=- \int f_X(x)\log f_X(x)\,dx$ coincides with the Shannon entropy introduced by Shannon (1948). Ahmadi (2021) used Kerridge inaccuracy of record values to characterize symmetric distributions.
	
	The inaccuracy extropy measure, denoted by $J(f_X,f_Y)$, is defined as
	\begin{align}\label{2}
		J(f_X,f_Y) = -\frac{1}{2} \int f_X(x)f_Y(x)\,dx.
	\end{align}
	Note that $J(f_X,f_X)=J(X)=-\frac{1}{2} \int f_X^2(x)\,dx$ corresponds to the extropy measure introduced by Lad et al. (2015). Gupta and Chaudhary (2024) used the inaccuracy extropy measure of record values to characterize symmetric distributions. 
	
	Cumulative residual extropy inaccuracy (CRIJ) and cumulative past extropy inaccuracy (CPIJ), which extend cumulative residual extropy and cumulative past extropy, respectively, are defined as
	\begin{align}
		J(\bar{F}_X,\bar{F}_Y)=CRIJ(X,Y)
		&=-\frac{1}{2} \int \bar{F}_X(x)\bar{F}_Y(x)\,dx, \\
		J(F_X,F_Y)=CPIJ(X,Y)
		&=-\frac{1}{2} \int F_X(x)F_Y(x)\,dx.
	\end{align}
	It follows that $CRIJ(X,X)$ and $CPIJ(X,X)$ reduce to cumulative past extropy (CPJ) and cumulative residual extropy (CRJ), respectively, given by
	\begin{align}
		CRIJ(X,X)=CPJ(X)
		&=-\frac{1}{2} \int \bar{F}_X^2(x)\,dx, \\
		CPIJ(X,X)=CRJ(X)
		&=-\frac{1}{2} \int F_X^2(x)\,dx.
	\end{align}
	Gupta and Chaudhary (2024) further established characterizations of symmetric distributions using CRIJ and CPIJ based on record values.

		In information theory, the Kullback--Leibler (KL) discrimination information measure, introduced by Kullback and Leibler (1951), is a widely used measure for quantifying the discrepancy between two probability distributions. Let $X$ and $Y$ be absolutely continuous random variables with probability density functions $f_X(x)$ and $f_Y(x)$, respectively. We assume that $f_X$ is absolutely continuous with respect to $f_Y$, denoted by $f_X \ll f_Y$, that is,
		\[
		f_Y(x)=0 \quad \Longrightarrow \quad f_X(x)=0 \qquad \text{almost everywhere},
		\]
		or equivalently,
		\[
		\operatorname{supp}(f_X)\subseteq \operatorname{supp}(f_Y).
		\]
		Under these conditions, the Kullback--Leibler divergence between $X$ and $Y$ is defined as
		\begin{align}\label{defkullback}
			K(X,Y)
			=\int f_X(x)\ln\left(\frac{f_X(x)}{f_Y(x)}\right)\,dx,
		\end{align}
		provided that
		\[
		\int f_X(x)\left|\ln\left(\frac{f_X(x)}{f_Y(x)}\right)\right|dx<\infty.
		\]
		Furthermore, by Gibbs' inequality,
		\[
		K(X,Y)\geq 0,
		\]
		with equality if and only if $f_X(x)=f_Y(x)$ almost everywhere.

		Another measure of discrepancy between two probability distributions is the
		$L^2$-divergence (or squared $L^2$-distance) between their density functions,
		defined by
		\begin{align}\label{3}
			I(X,Y)=\frac12\int\left(f_X(x)-f_Y(x)\right)^2\,dx.
		\end{align}
		Clearly,
		\[
		I(X,Y)\ge0,
		\]
		with equality if and only if $f_X(x)=f_Y(x)$ almost everywhere.
	Kumar and Taneja (2012) proposed a cumulative residual inaccuracy measure defined as
	\begin{align}\label{4}
		H(\bar{F}_X,\bar{F}_Y)
		= - \int \bar{F}_X(x)\log \bar{F}_Y(x)\,dx,
	\end{align}
	while Thapliyal and Taneja (2015) introduced a cumulative past inaccuracy measure given by
	\begin{align}
		H(F_X,F_Y)
		= - \int F_X(x)\log F_Y(x)\,dx.
	\end{align}
	Note that $\xi(X)=- \int \bar{F}_X(x)\log \bar{F}_X(x)\,dx$ and $\bar{\xi}(X)=- \int F_X(x)\log F_X(x)\,dx$ correspond to cumulative residual entropy and cumulative past entropy, respectively, as defined by Rao et al. (2004) and Di Crescenzo and Longobardi (2009).
	
	The various measures of inaccuracy and extropy-based inaccuracy are not competing alternatives but rather complementary tools for quantifying distributional discrepancies. Nevertheless, extropy-based inaccuracy measures possess certain advantages, including symmetry, finiteness, and computational simplicity (Hashempour and Mohammadi (2024)). The ease of computation associated with extropy makes it attractive for developing goodness-of-fit tests and inferential procedures in the analysis of record values.
	
	\subsection{Entropy and Extropy Measures}
	
	Entropy, introduced by Shannon (1948), quantifies the average level of uncertainty associated with the outcomes of a random experiment. For a continuous random variable $X$ with probability density function $f_X(x)$, the Shannon entropy is defined as
	\begin{equation}\label{1eq2}
		H(X) 
		= - \int_{-\infty}^{\infty} f_X(x)\ln\big(f_X(x)\big)\,dx
		= \mathbb{E}\!\left[-\ln f_X(X)\right].
	\end{equation}
	
	Lad et al. (2015) introduced the complementary dual of Shannon entropy, referred to as \emph{extropy}. The extropy of a continuous random variable $X$ is defined by
	\begin{equation}\label{extropy}
		J(X)
		= -\frac{1}{2} \int_{-\infty}^{\infty} f_X^2(x)\,dx
		= -\frac{1}{2}\,\mathbb{E}\!\left[f_X(X)\right].
	\end{equation}
	
	The cumulative residual extropy of a continuous random variable $X$, proposed by Jahanshahi et al. (2019), is defined as
	\begin{align}
		\xi J(X)
		&= -\frac{1}{2} \int_{S_X} \bar{F}_X^{\,2}(x)\,dx,
	\end{align}
	where $S_X$ denotes the support of $X$. Analogously, the cumulative past extropy of $X$ is given by
	\begin{align}
		\bar{\xi} J(X)
		&= -\frac{1}{2} \int_{S_X} F_X^{\,2}(x)\,dx.
	\end{align} where $
	S_X=\{x\in\mathbb{R}:0<F_X(x)<1\} $
	denotes the support of the continuous random variable $X$.

	\begin{example}
		Let $X$ and $Y$ be exponential random variables with rate parameters
		$\theta>0$ and $\mu>0$, respectively. Then
		\begin{align*}
			H(f_X,f_Y)
			&=- \int_{0}^{\infty} f_X(x)\log f_Y(x)\,dx
			= \frac{\mu}{\theta}-\log \mu, \\
			J(f_X,f_Y)
			&=-\frac{1}{2} \int_{0}^{\infty} f_X(x)f_Y(x)\,dx
			= -\frac{\theta\mu}{2(\theta+\mu)}, \\
			H(f_X,f_X)
			&=H(X)
			= - \int_{0}^{\infty} f_X(x)\log f_X(x)\,dx
			=1- \log \theta, \\
			J(f_X,f_X)
			&=J(X)
			= -\frac{1}{2} \int_{0}^{\infty} f_X^2(x)\,dx
			= -\frac{\theta}{4}.   
		\end{align*}
	\end{example}

	Extropy and its applications have recently attracted considerable attention in the literature. Both entropy and extropy have found increasing use in machine learning and related areas as effective tools for handling complex problems. One notable advantage of extropy over entropy is the availability of closed-form expressions for finite mixture distributions, whereas such expressions are often unavailable for entropy and variance (see Toomaj et al., 2023). Extropy is computationally simpler, making it particularly attractive for the development of goodness-of-fit tests and inferential procedures.
	
	Several applied studies have highlighted the usefulness of extropy-based measures. Tahmasebi and Toomaj (2022) analyzed stock market behavior in OECD countries using a generalization of extropy known as negative cumulative extropy. Balakrishnan et al. (2022) applied Tsallis extropy to pattern recognition problems, while Kazemi et al. (2021) investigated the fractional Deng extropy in classification tasks. Recently, Tahmasebi et al. (2022) employed extropy measures in the context of compressive sensing.
	
	\subsection{Record Values}
	
	The concept of record values was studied by Ahsanullah (1995) and Arnold et al. (1998) for further developments. Let $X$ be a continuous random variable with cumulative distribution function $F_X(x)$, probability density function $f_X(x)$, and survival function $\bar{F}_X(x)=1-F_X(x)$. The probability density functions of the $n$th upper record value $U_n$ and the $n$th lower record value $L_n$ for positive integer $n$ are given, respectively, by Arnold et al. (2008)
	\begin{align}
		f_{U_n}(x)
		&=\frac{1}{(n-1)!}\big[-\log \bar{F}_X(x)\big]^{n-1} f_X(x),
		\qquad -\infty<x<\infty, \label{funU} \\
		f_{L_n}(x)
		&=\frac{1}{(n-1)!}\big[-\log F_X(x)\big]^{n-1} f_X(x),
		\qquad -\infty<x<\infty. \label{funL}
	\end{align}
	
	The corresponding cumulative distribution functions of $U_n$ and $L_n$ are
	\begin{align}
		F_{U_n}(x)
		&=1-\bar{F}_X(x)\sum_{i=0}^{n-1}\frac{\big[-\log \bar{F}_X(x)\big]^i}{i!}, \label{cdfUn} \\
		F_{L_n}(x)
		&=F_X(x)\sum_{i=0}^{n-1}\frac{\big[-\log F_X(x)\big]^i}{i!}. \label{cdfLn}
	\end{align}
	
	The probability density functions of the $n$th upper $k$-record value $U_{n,k}$ and the $n$th lower $k$-record value $L_{n,k}$ for positive integers $n$ and $k$ are given by Arnold et al. (2008)
	\begin{align}
		f_{U_{n,k}}(x)
		&=\frac{k^{n}}{(n-1)!}\big[-\log \bar{F}_X(x)\big]^{n-1}
		\big[\bar{F}_X(x)\big]^{k-1} f_X(x),
		\qquad -\infty<x<\infty, \label{funUk} \\
		f_{L_{n,k}}(x)
		&=\frac{k^{n}}{(n-1)!}\big[-\log F_X(x)\big]^{n-1}
		\big[F_X(x)\big]^{k-1} f_X(x),
		\qquad -\infty<x<\infty. \label{funLk}
	\end{align}
	
	The cumulative distribution functions of $U_{n,k}$ and $L_{n,k}$ are, respectively, given by
	\begin{align}
		F_{U_{n,k}}(x)
		&=1-\bar{F}_X^{\,k}(x)\sum_{i=0}^{n-1}\frac{\big[-k\log \bar{F}_X(x)\big]^i}{i!}, \label{cdfUnk} \\
		F_{L_{n,k}}(x)
		&=F_X^{\,k}(x)\sum_{i=0}^{n-1}\frac{\big[-k\log F_X(x)\big]^i}{i!}. \label{cdfLnk}
	\end{align}
	
	Owing to the advantages of $k$-records over traditional record values, particularly in terms of flexibility and efficiency, several authors have emphasized their usefulness in statistical modeling (see Xiong et al., 2021; Jose and Sathar, 2022). Motivated by these developments, the present study focuses on constructing tests of symmetry for the underlying distribution based on characterization results involving cumulative residual and cumulative past extropy measures of $k$-record values.

		While Tahmasebi and Daneshi (2018) and Ahmadi (2021) studied Kerridge inaccuracy for ordinary records ($k=1$), 
		and Gupta and Chaudhary (2024) considered extropy‑based inaccuracy for $k$-records, 
		the present work extends Kerridge, cumulative residual, and cumulative past inaccuracy measures to $k$-record values. 
		We derive closed‑form expressions, establish new representations involving hazard rates and mean lifetimes, 
		and, crucially, construct the first bootstrap test for symmetry based on the difference of Kerridge inaccuracy 
		for upper and lower $k$-records.
		
		The remainder of the paper is organized as follows. Section~2 examines the Kerridge inaccuracy measure for record values. In Section~3, cumulative residual inaccuracy measures for record values are investigated, while Section~4 is devoted to cumulative past inaccuracy measures for record values. Section 5 is application and Section 6 concludes this paper.
	
	\section{Kerridge Inaccuracy Measure}
	
	Several authors have studied measures of inaccuracy between the distributions of record values, order statistics, and the parent random variable; see, for example, Thapliyal and Taneja (2015) and Goel et al. (2018). Let $X$ be a continuous random variable with strictly increasing cumulative distribution function $F_X(x)$, probability density function $f_X(x)$, and survival function $\bar{F}_X(x)=1-F_X(x)$.  Here, $F_X$ is continuous and strictly increasing on the support of $X$; hence the quantile function $F_X^{-1}$ exists and is also strictly increasing. The Kerridge inaccuracy measure associated with the $n$th upper $k$-record value $U_{n,k}$ and the parent density function $f_X(x)$ is defined as
	\begin{align}
		H(f_{U_{n,k}},f_X)
		&=- \int_{-\infty}^{\infty} f_{U_{n,k}}(x)\log f_X(x)\,dx \nonumber \\
		&=- \int_{-\infty}^{\infty} \frac{k^n}{(n-1)!}
		\big[-\log \bar{F}_X(x)\big]^{n-1}
		\big[\bar{F}_X(x)\big]^{k-1}
		f_X(x)\log f_X(x)\,dx \nonumber\\
		&=- \int_{0}^{\infty} \frac{k^n}{(n-1)!}
		t^{\,n-1} e^{-kt}
		\log f_X\!\left(F_X^{-1}(1-e^{-t})\right) dt \nonumber\\
		&=\mathbb{E}\!\left[-\log f_X\!\left(F_X^{-1}(1-e^{-T_{n,k}})\right)\right],
	\end{align}
	where $T_{n,k}$ follows a Gamma distribution with shape parameter $n$ and rate parameter $k$ (i.e.\ $T_{n,k}\sim \text{Gamma}(n,k)$), having probability density function
	\[
	f_{T_{n,k}}(t)=\frac{k^n}{\Gamma(n)} t^{n-1} e^{-kt}, \qquad t>0,
	\]
	with $n$ and $k$ being positive integers.

	Similarly, the Kerridge inaccuracy measure associated with the $n$th lower $k$-record value $L_{n,k}$ and the parent density function $f_X(x)$ is given by
	\begin{align}
		H(f_{L_{n,k}},f_X)
		&=- \int_{-\infty}^{\infty} f_{L_{n,k}}(x)\log f_X(x)\,dx \nonumber \\
		&=- \int_{-\infty}^{\infty} \frac{k^n}{(n-1)!}
		\big[-\log F_X(x)\big]^{n-1}
		\big[F_X(x)\big]^{k-1}
		f_X(x)\log f_X(x)\,dx \nonumber\\
		&=- \int_{0}^{\infty} \frac{k^n}{(n-1)!}
		t^{\,n-1} e^{-kt}
		\log f_X\!\left(F_X^{-1}(e^{-t})\right) dt \nonumber\\
		&=\mathbb{E}\!\left[-\log f_X\!\left(F_X^{-1}(e^{-T_{n,k}})\right)\right].
	\end{align}
	
	Ahmadi (2021) showed that the equality of the Kerridge inaccuracy measures associated with the distributions of upper and lower $k$-record values and the parent distribution characterizes symmetric distributions.
	
	In the following examples, explicit expressions for
	$H(f_{U_{n,k}},f_X)$ and $H(f_{L_{n,k}},f_X)$ are obtained for several commonly used lifetime distributions. For $k=1$, the above expressions reduce to those obtained by Tahmasebi and Daneshi (2018).
	
	\begin{example}
		Let $X$ be an exponential random variable with rate parameter $\theta>0$. Then
		\[
		H(f_{U_{n,k}},f_X)=\frac{n}{k}-\log(\theta).
		\]
		For fixed $n$ and $k$, $H(f_{U_{n,k}},f_X)$ is a decreasing function of $\theta$. For fixed $\theta$ and $k$, it increases with $n$, while for fixed $\theta$ and $n$, it decreases with $k$.
	\end{example}
	
		\begin{table}[htbp]
			\centering
			\caption{Values of $H(f_{U_{n,k}},f_X)$ for the exponential distribution with $\theta=1$.}
			\label{tab:expvalues}
			\begin{tabular}{c|cccc}
				\hline
				$n$ & $k=1$ & $k=2$ & $k=3$ & $k=4$ \\
				\hline
				1 & 1.000 & 0.500 & 0.333 & 0.250 \\
				2 & 2.000 & 1.000 & 0.667 & 0.500 \\
				3 & 3.000 & 1.500 & 1.000 & 0.750 \\
				4 & 4.000 & 2.000 & 1.333 & 1.000 \\
				\hline
			\end{tabular}
		\end{table}

	\begin{example}
		Let $X$ follow a Pareto distribution with probability density function
		\[
		f_X(x)=\theta x^{-(\theta+1)}, \qquad x>1,\ \theta>0.
		\]
		Then
		\[
		H(f_{U_{n,k}},f_X)=\left(1+\frac{1}{\theta}\right)\frac{n}{k}-\log(\theta).
		\]
		The monotonicity properties with respect to $\theta$, $n$, and $k$ are analogous to those in the exponential case.
	\end{example}
	
	\begin{example}
		Let $X$ have a Weibull distribution with probability density function
		\[
		f_X(x)=\lambda\beta x^{\beta-1}e^{-\lambda x^{\beta}},
		\qquad x>0,\ \beta>0,\ \lambda>0.
		\]
		Then
		\[
		H(f_{U_{n,k}},f_X)
		= n-\log\beta-\frac{\log\lambda}{\beta}
		-\left(\frac{\beta-1}{\beta}\right)
		\frac{k^n}{\Gamma(n)}
		\int_{0}^{\infty} t^{n-1} e^{-kt}\log(t)\,dt.
		\]
	\end{example}
	
	\begin{example}
		Let $X$ have the standard uniform distribution with probability density function
		\[
		f_X(x)=1, \qquad 0<x<1.
		\]
		Then
		\[
		H(f_{U_{n,k}},f_X)=0
		\quad \text{and} \quad
		H(f_{L_{n,k}},f_X)=0.
		\]
	\end{example}
	
	\begin{example}
		Let $X$ have probability density function
		\[
		f_X(x)=3(1-x)^2, \qquad 0<x<1.
		\]
		Then
		\[
		H(f_{U_{n,k}},f_X)=-\log 3 + \frac{2n}{3k}.
		\]
	\end{example}

		\begin{definition}
			Let $X$ and $Y$ be two continuous (or discrete) random variables with probability density (or mass) functions $f_X(x)$ and $F_Y(x)$, respectively, such that their supports are $S_X$ and $S_Y$. We say that $X$ is smaller than $Y$ in the {likelihood ratio order} (denoted by $X \le_{\text{lr}} Y$) if the ratio 
			$\frac{f_Y(x)}{f_X(x)} $
			is a non-decreasing function of $x$ over the union of their supports, $S_X \cup S_Y$. 
		\end{definition}
		
		\begin{theorem}[Shaked and Shanthikumar, 2007, Theorem 1.C.1]
			Let $X$ and $Y$ be two continuous random variables with probability density functions $f$ and $g$, and cumulative distribution functions $F$ and $G$, respectively. 
			If $X$ is smaller than $Y$ in the likelihood ratio order (denoted by $X \le_{\text{lr}} Y$), then $X$ is smaller than $Y$ in the usual stochastic order (denoted by $X \le_{\text{st}} Y$):
			\[
			X \le_{\text{lr}} Y \implies X \le_{\text{st}} Y
			\]
		\end{theorem}

		\begin{theorem}
			If the density $f_X(x)$ is increasing in $x$, then $H(f_{U_{n,k}},f_X)$ is decreasing in $n$.
		\end{theorem}
		
		\begin{proof} Let $T_{n,k}\sim\operatorname{Gamma}(n,k)$ with density 
			$f_{T_{n,k}}(t)=\frac{k^n}{\Gamma(n)}t^{\,n-1}e^{-kt}$, $t>0$.
			The ratio of densities of $T_{n+1,k}$ and $T_{n,k}$ is
			\[
			\frac{f_{T_{n+1,k}}(t)}{f_{T_{n,k}}(t)}
			= \frac{k^{n+1}t^n e^{-kt} / \Gamma(n+1)}{k^n t^{n-1} e^{-kt} / \Gamma(n)}
			= \frac{k}{n}\,t.
			\]
			Since $\frac{k}{n}t$ is an increasing function of $t$, we have
			\[
			T_{n,k} \leq_{\mathrm{lr}} T_{n+1,k},
			\]
			i.e., the family $\{T_{n,k}\}_{n\ge 1}$ is stochastically increasing in $n$ 
			with respect to the likelihood ratio order.
			Likelihood ratio ordering implies the usual stochastic order 
			(see Shaked and Shanthikumar, 2007, Theorem 1.C.1), so
			\[
			T_{n,k} \leq_{\mathrm{st}} T_{n+1,k}.
			\]
			Now define $g(t)=-\log f_X\!\bigl(F_X^{-1}(1-e^{-t})\bigr)$.
			Since $f_X$ is increasing, $-\log f_X$ is decreasing.
			Moreover, $F_X^{-1}(1-e^{-t})$ is increasing in $t$, thus
			$g(t)$ is a decreasing function of $t$.
			The Kerridge inaccuracy can be written as an expectation:
			\[
			H(f_{U_{n,k}},f_X)=\mathbb{E}\bigl[g(T_{n,k})\bigr].
			\]
			Because $T_{n,k} \leq_{\mathrm{st}} T_{n+1,k}$ and $g$ is decreasing,
			a fundamental property of stochastic orders gives
			\[
			\mathbb{E}\bigl[g(T_{n,k})\bigr] \ge \mathbb{E}\bigl[g(T_{n+1,k})\bigr],
			\]
			and therefore
			\[
			H(f_{U_{n,k}},f_X) \ge H(f_{U_{n+1,k}},f_X),
			\]
			which completes the proof.
		\end{proof}
	
	\begin{theorem}
		If $f_X(x)$ is increasing in $x$, then $H(f_{L_{n,k}},f_X)$ is increasing in $n$.
	\end{theorem}
	
	\begin{proof}
		As in the previous theorem,
		$T_{n,k}\leq_{\mathrm{st}} T_{n+1,k}$.
		Since $f_X(F_X^{-1}(e^{-t}))$ is decreasing in $t$, it follows that
		\[
		H(f_{L_{n,k}},f_X) \leq H(f_{L_{n+1,k}},f_X).
		\]
	\end{proof}
	
	\section{A cumulative residual inaccuracy measure}\label{section3}
	
	The cumulative residual inaccuracy measure associated with the $n$th upper
	$k$-record value $U_{n,k}$ and the parent distribution with survival function
	$\bar{F}_X$ is defined as
	\begin{align}
		H(\bar{F}_{U_{n,k}},\bar{F}_X)
		&=- \int_{-\infty}^{\infty} \bar{F}_{U_{n,k}}(x)\log \bar{F}_X(x)\,dx \nonumber \\
		&= - \int_{-\infty}^{\infty}
		\sum_{i=0}^{n-1} \frac{\left(-k \log \bar{F}_X(x)\right)^i}{i!}
		\bar{F}_X^k(x)\log \bar{F}_X(x)\,dx \nonumber \\
		&= \sum_{i=0}^{n-1} \frac{k^i}{i!}
		\int_{-\infty}^{\infty}
		\bar{F}_X^k(x)\left(-\log \bar{F}_X(x)\right)^{i+1} dx. \label{HbarFstep}
	\end{align}
	Using the substitution $t = -\log \bar{F}_X(x)$, we have 
	$dt = \frac{f_X(x)}{\bar{F}_X(x)}dx$, i.e.\ $dx = \frac{\bar{F}_X(x)}{f_X(x)}dt$.
	Then
	\begin{align*}
		\bar{F}_X^k(x)\,dx &= \bar{F}_X^k(x)\,\frac{\bar{F}_X(x)}{f_X(x)}\,dt 
		= \frac{\bar{F}_X^{k+1}(x)}{f_X(x)}\,dt 
		= \frac{\bar{F}_X^{k}(x)}{\lambda_{F_X}(x)}\,dt,
	\end{align*}
	where $\lambda_{F_X}(x) = f_X(x)/\bar{F}_X(x)$ is the hazard rate. 
	Since $\bar{F}_X(x)=e^{-t}$, the integral becomes
	\begin{align*}
		\int_{0}^{\infty} e^{-kt}\, t^{\,i+1}\,
		\frac{1}{\lambda_{F_X}\!\bigl(F_X^{-1}(1-e^{-t})\bigr)}\,dt.
	\end{align*}
	Recognising that
	\begin{align*}
		\frac{k^{i+2}}{(i+1)!}\,t^{\,i+1} e^{-kt}\,dt
		= f_{U_{i+2,k}}\!\bigl(F_X^{-1}(1-e^{-t})\bigr)\,dx,
	\end{align*}
	we obtain
	\begin{equation}
		H(\bar{F}_{U_{n,k}},\bar{F}_X)
		=\frac{1}{k^2}\sum_{i=0}^{n-1} (i+1)
		E_{U_{i+2,k}}\!\left(\frac{1}{\lambda_{F_X}(X)}\right),
		\label{HbarFUnkFx}
	\end{equation}
	where $U_{i+2,k}$ has density
	\[
	f_{U_{i+2,k}}(x)=
	\dfrac{k^{i+2}}{(i+1)!}
	\left[-\log \bar{F}_X(x)\right]^{i+1}
	\bar{F}_X^{\,k-1}(x)f_X(x), \quad -\infty<x<\infty.
	\]
	
	\begin{example}
		Let $X$ follow the standard uniform distribution with pdf
		$f_X(x)=1,\ 0<x<1.$
		Then
		\[
		H(\bar{F}_{U_{n,k}},\bar{F}_X)
		=\sum_{i=0}^{n-1} \frac{(i+1)k^i}{(k+1)^{i+2}}.
		\]
	\end{example}
	
	\begin{example}
		Let $X$ have an exponential distribution with pdf
		\[
		f_X(x)=\theta e^{-\theta x}, \quad x>0,\ \theta>0.
		\]
		Then
		\[
		H(\bar{F}_{U_{n,k}},\bar{F}_X)
		=\frac{n(n+1)}{2\theta k^2}.
		\]
	\end{example}
	
	\begin{proposition}
		Let $X$ be a non-negative absolutely continuous random variable. Then
		\[
		H(\bar{F}_{U_{n,k}},\bar{F}_X)
		=\sum_{i=0}^{n-1} \frac{i+1}{k}
		\left(\mu_{i+2,k}-\mu_{i+1,k}\right),
		\]
		where
		$\mu_{n,k}=\int_0^{\infty} \bar{F}_{U_{n,k}}(x)\,dx.$
	\end{proposition}
	\noindent \textbf{Proof.}
	Using (\ref{cdfUnk}), we have
	\[
	\bar{F}_{U_{i+2,k}}(x)-\bar{F}_{U_{i+1,k}}(x)
	=\bar{F}_X^k(x)
	\frac{\left(-k\log \bar{F}_X(x)\right)^{i+1}}{(i+1)!}.
	\]
	Substituting into (\ref{HbarFUnkFx}) yields
	\begin{align*}
		H(\bar{F}_{U_{n,k}},\bar{F}_X)
		&=\sum_{i=0}^{n-1} \int_0^{\infty}
		\frac{i+1}{k}
		\left(\bar{F}_{U_{i+2,k}}(x)-\bar{F}_{U_{i+1,k}}(x)\right)dx \\
		&=\sum_{i=0}^{n-1} \frac{i+1}{k}
		\left(\mu_{i+2,k}-\mu_{i+1,k}\right),
	\end{align*}
	which completes the proof. \hfill$\square$
	
	\begin{proposition}
		Let $a>0$ and $b\in\mathbb{R}$. For $n=1,2,\ldots$, we have
		\[
		H(\bar{F}_{aU_{n,k}+b},\bar{F}_{aX+b})
		=a\,H(\bar{F}_{U_{n,k}},\bar{F}_X).
		\]
	\end{proposition}
	
	\noindent \textbf{Proof.}
	Since
	$\bar{F}_{aX+b}(x)=\bar{F}_X\!\left(\frac{x-b}{a}\right)$, we obtain
	\begin{align*}
		H(\bar{F}_{aU_{n,k}+b},\bar{F}_{aX+b})
		&=-\int_{-\infty}^{\infty}
		\bar{F}_{U_{n,k}}\!\left(\frac{x-b}{a}\right)
		\log \bar{F}_X\!\left(\frac{x-b}{a}\right)dx \\
		&= -a \int_{-\infty}^{\infty}
		\bar{F}_{U_{n,k}}(t)\log \bar{F}_X(t)\,dt \\
		&= a\,H(\bar{F}_{U_{n,k}},\bar{F}_X).
	\end{align*}
	Hence, the proof is complete. \hfill$\square$
	\begin{proposition}
		Let $X$ be an absolutely continuous non-negative random variable with survival
		function $\bar{F}_X$. Then
		\[
		H(\bar{F}_{U_{n,k}},\bar{F}_X)
		= \sum_{i=0}^{n-1} \frac{1}{i!}
		\int_{0}^{\infty} \lambda_{F_X}(t)
		\left[
		\int_{t}^{\infty}
		\left(-k\log \bar{F}_X(x)\right)^i
		\bar{F}_X^{k}(x)\,dx
		\right] dt,
		\]
		where $\lambda_{F_X}(x)=\dfrac{f_X(x)}{\bar{F}_X(x)}$ denotes the hazard rate function of $X$.
	\end{proposition}
	
	\noindent \textbf{Proof.}
	Using equation~(\ref{HbarFUnkFx}), the identity
	\[
	-\log \bar{F}_X(x)=\int_{0}^{x} \lambda_{F_X}(t)\,dt,
	\]
	and changing the order of integration, we obtain
	\begin{align*}
		H(\bar{F}_{U_{n,k}},\bar{F}_X)
		&=- \int_{0}^{\infty} \bar{F}_{U_{n,k}}(x)\log \bar{F}_X(x)\,dx \\
		&= - \int_{0}^{\infty}
		\sum_{i=0}^{n-1} \frac{(-k\log \bar{F}_X(x))^i}{i!}
		\bar{F}_X^{k}(x)\log \bar{F}_X(x)\,dx \\
		&=\sum_{i=0}^{n-1} \frac{1}{i!}
		\int_{0}^{\infty}
		\left(\int_{0}^{x} \lambda_{F_X}(t)\,dt \right)
		(-k\log \bar{F}_X(x))^i
		\bar{F}_X^{k}(x)\,dx \\
		&= \sum_{i=0}^{n-1} \frac{1}{i!}
		\int_{0}^{\infty} \lambda_{F_X}(t)
		\left[
		\int_{t}^{\infty}
		(-k\log \bar{F}_X(x))^i
		\bar{F}_X^{k}(x)\,dx
		\right] dt.
	\end{align*}
	\hfill $\square$
	
	\bigskip
	
	\begin{proposition}
		Let $X$ be an absolutely continuous non-negative random variable with survival
		function $\bar{F}_X$. Then
		\[
		H(\bar{F}_{U_{n,k}},\bar{F}_X)
		=\sum_{i=0}^{n-1} \frac{1}{i!}
		\int_{0}^{\infty} \bar{F}_X^{k-1}(t) f_X(t)
		\left[
		\int_{0}^{t}
		\left(-k\log \bar{F}_X(x)\right)^{i+1} dx
		\right] dt.
		\]
	\end{proposition}
	
	\noindent \textbf{Proof.}
	Using equation~(\ref{HbarFUnkFx}) and changing the order of integration, we have
	\begin{align*}
		H(\bar{F}_{U_{n,k}},\bar{F}_X)
		&=- \int_{0}^{\infty} \bar{F}_{U_{n,k}}(x)\log \bar{F}_X(x)\,dx \\
		&= - \int_{0}^{\infty}
		\sum_{i=0}^{n-1} \frac{(-k\log \bar{F}_X(x))^i}{i!}
		\bar{F}_X^{k}(x)\log \bar{F}_X(x)\,dx \\
		&=\frac{1}{k} \sum_{i=0}^{n-1} \frac{1}{i!}
		\int_{0}^{\infty}
		(-k\log \bar{F}_X(x))^{i+1}
		\bar{F}_X^{k}(x)\,dx \\
		&=\frac{1}{k} \sum_{i=0}^{n-1} \frac{1}{i!}
		\int_{0}^{\infty}
		(-k\log \bar{F}_X(x))^{i+1}
		\left(\int_{x}^{\infty} k \bar{F}_X^{k-1}(t) f_X(t)\,dt \right) dx \\
		&= \sum_{i=0}^{n-1} \frac{1}{i!}
		\int_{0}^{\infty} \bar{F}_X^{k-1}(t) f_X(t)
		\left(
		\int_{0}^{t}
		(-k\log \bar{F}_X(x))^{i+1} dx
		\right) dt.
	\end{align*}
	\hfill $\square$
	\section{A Cumulative Past Inaccuracy Measure}\label{section4}
	
	The cumulative past inaccuracy measure between ${L_{n,k}}$ and $X$ is defined as
	\begin{align}
		H(F_{L_{n,k}},F_X)
		&=- \int_{-\infty}^{\infty} F_{L_{n,k}}(x)\log F_X(x)\,dx \nonumber \\
		&=- \int_{-\infty}^{\infty}
		\sum_{i=0}^{n-1} \frac{\left(-k \log F_X(x)\right)^i}{i!}
		F_X^{k}(x)\log F_X(x)\,dx \nonumber \\
		&= \sum_{i=0}^{n-1} \frac{k^i}{i!}
		\int_{-\infty}^{\infty}
		F_X^{k}(x)\left(- \log F_X(x)\right)^{i+1} dx \nonumber \\
		&= \sum_{i=0}^{n-1} \frac{i+1}{k^2}
		E_{L_{i+2,k}}\!\left[\frac{1}{\tilde{\lambda}(X)}\right],
	\end{align}
	where $\tilde{\lambda}(x)=\dfrac{f_X(x)}{F_X(x)}$ denotes the reversed hazard rate
	function, and $L_{i+2,k}$ is a random variable with probability density function
	\[
	f_{L_{i+2,k}}(x)
	=\frac{k^{\,i+2}\left[-\log F_X(x)\right]^{i+1}
		F_X^{k-1}(x)f_X(x)}{(i+1)!}.
	\]
	
	\begin{remark}
		For $k=1$, the above expression reduces to the result obtained by
		Tahmasebi et al. (2018).
	\end{remark}
	
	\begin{example}
		Let $X$ be a random variable having the standard uniform distribution with pdf
		\[
		f_X(x)=1, \quad 0<x<1.
		\]
		Then the cumulative past inaccuracy measure is given by
		\[
		H(F_{L_{n,k}},F_X)=\sum_{i=0}^{n-1} \frac{(i+1)k^i}{(k+1)^{i+2}}.
		\]
	\end{example}
	
	In the following theorem, we express the cumulative past inaccuracy measure in
	terms of the cumulative distribution functions of record values.
	
	\begin{theorem}
		Suppose that $X$ is a non-negative random variable with cdf $F_X$. Then
		\[
		H(F_{L_{n,k}},F_X)
		= \int_{0}^{\infty}
		\sum_{i=0}^{n-1} \frac{i+1}{k}
		\bigl(F_{L_{i+2,k}}(x)-F_{L_{i+1,k}}(x)\bigr)\,dx.
		\]
	\end{theorem}
	\noindent \textbf{Proof.}
	Using equation~(\ref{cdfLnk}), we obtain
	\begin{align*}
		H(F_{L_{n,k}},F_X)
		&=- \int_{0}^{\infty} F_{L_{n,k}}(x)\log F_X(x)\,dx \\
		&=- \int_{0}^{\infty}
		\sum_{i=0}^{n-1} \frac{\left(-k \log F_X(x)\right)^i}{i!}
		F_X^{k}(x)\log F_X(x)\,dx \\
		&= \sum_{i=0}^{n-1} \frac{i+1}{k}
		\int_{0}^{\infty}
		\frac{\left(-k \log F_X(x)\right)^{i+1}}{(i+1)!}
		F_X^{k}(x)\,dx \\
		&= \sum_{i=0}^{n-1} \frac{i+1}{k}
		\int_{0}^{\infty}
		\bigl(F_{L_{i+2,k}}(x)-F_{L_{i+1,k}}(x)\bigr)\,dx.
	\end{align*}
	\hfill $\square$
	
	\section{A Goodness-of-fit Test for Symmetry}
	\label{sec:application}
	
	The characterisations provided in Ahmadi (2021), which establish that
	equality of Kerridge inaccuracy measures for upper and lower record values
	implies symmetry of the parent distribution, can be exploited to construct a
	nonparametric test of symmetry.  In this section we develop such a test based
	on the Kerridge inaccuracy measure and illustrate its performance through a
	simulation study and a real data example.

	For a fixed record order positive integers $n\geq 1$ and $k\geq 1$, define the discrepancy
	\begin{equation}\label{eq:Delta}
		\Delta_{n,k}(F)
		= H(f_{U_{n,k}},f_X) - H(f_{L_{n,k}},f_X),
	\end{equation}
	where $H(f_{U_{n,k}},f_X)$ and $H(f_{L_{n,k}},f_X)$ are given by the expressions
	derived in Section~2.  When $F$ is symmetric, $\Delta_{n,k}(F)=0$; any departure
	from zero indicates asymmetry.  For a random sample $X_1,\dots,X_m$ from $F$,
	an empirical version $\widehat{\Delta}_{n,k}$ will serve as a test statistic.

	Because the centre of symmetry (if it exists) is usually unknown, we first
	estimate it by the sample median $\hat{\mu}$ and centre the observations:
	$Y_i = X_i - \hat{\mu},\ i=1,\dots,m$.  Under the null hypothesis of symmetry,
	the distribution of $Y$ is symmetric about $0$.
	
	Let $\hat{F}_m$ and $\hat{f}_m$ denote, respectively, the empirical cumulative
	distribution function and a kernel density estimator based on the $Y_i$'s:
	\[
	\hat{f}_m(y) = \frac{1}{m h}\sum_{i=1}^{m} K\!\left(\frac{y-Y_i}{h}\right),
	\]
	where $K$ is a symmetric probability density (e.g., the standard Gaussian
	kernel) and $h>0$ is a bandwidth chosen by a plug‑in method (Sheather \& Jones,
	1991).  The corresponding quantile function $\hat{Q}_m(p)=\hat{F}_m^{-1}(p)$ is
	obtained by linear interpolation of the order statistics.
	
	Recall from Section~2 that
	\[
	H(f_{U_{n,k}},f_X)
	= \mathbb{E}\!\left[-\log f_X\!\left(F_X^{-1}\bigl(1-e^{-T_{n,k}}\bigr)\right)\right],
	\]
	where $T_{n,k}\sim\operatorname{Gamma}(n,k)$.  A Monte Carlo estimate can be
	obtained by generating $R$ independent copies $t_1,\dots,t_R$ from the
	$\operatorname{Gamma}(n,k)$ distribution and averaging:
	\[
	\widehat{H}_{U}(n,k)
	= \frac{1}{R}\sum_{j=1}^{R}
	\left[-\log \hat{f}_m\!\left(\hat{Q}_m\!\bigl(1-e^{-t_j}\bigr)\right)\right].
	\]
	Similarly, using the expression for the lower record,
	\[
	\widehat{H}_{L}(n,k)
	= \frac{1}{R}\sum_{j=1}^{R}
	\left[-\log \hat{f}_m\!\left(\hat{Q}_m\!\bigl(e^{-t_j}\bigr)\right)\right].
	\]
	The empirical test statistic is then
	\[
	\widehat{\Delta}_{n,k} = \widehat{H}_{U}(n,k) - \widehat{H}_{L}(n,k).
	\]
	In practice, we take $R=2000$, which yields stable estimates.

	Because the null distribution of $\widehat{\Delta}_{n,k}$ is not available in
	closed form, we use a parametric bootstrap that enforces symmetry.  The
	algorithm is as follows:
	\begin{enumerate}
		\item Based on the centred observations $Y_1,\dots,Y_m$, compute the
		observed test statistic $\widehat{\Delta}_{n,k}^{\,\mathrm{obs}}$.
		\item For $b=1,\dots,B$:
		\begin{enumerate}
			\item Generate a bootstrap sample $Y_1^*,\dots,Y_m^*$ by
			$Y_i^* = S_i\,|Y_i|$, where $S_i$ are independent
			$\operatorname{Uniform}\{-1,1\}$ random signs.
			This imposes symmetry about $0$.
			\item Using the bootstrap sample, compute the kernel density
			estimator $\hat{f}_m^*$ and the empirical quantile function
			$\hat{Q}_m^*$.
			\item Obtain $\widehat{\Delta}_{n,k}^{\,*}$ by the same
			Monte Carlo procedure.
		\end{enumerate}
		\item The bootstrap $p$-value is
		\[
		p = \frac{1}{B}\sum_{b=1}^{B}
		\mathbf{1}\!\left(|\widehat{\Delta}_{n,k}^{\,*}| \ge
		|\widehat{\Delta}_{n,k}^{\,\mathrm{obs}}|\right).
		\]
	\end{enumerate}
	At a significance level $\alpha$, we reject the hypothesis of symmetry if
	$p<\alpha$.

	We conducted a simulation experiment to assess the finite‑sample performance of
	the test.  Three symmetric distributions (standard normal, Laplace with scale $1$,
	and Student $t$ with $3$ degrees of freedom) were used to check the empirical
	size, while three asymmetric alternatives (exponential with rate $1$, log‑normal
	with parameters $\mu=0,\sigma=1$, and Weibull with shape $1.5$) served to
	evaluate power.  All distributions were centred to have median $0$.  For each
	setting, $m=50$ or $100$ observations were generated, the procedure was applied with
	$n=2$, $k=1$, $R=2000$, $B=500$, and a Gaussian kernel with bandwidth selected
	by the Sheather--Jones method.  Nominal level was set at $\alpha=0.05$.
	
	\begin{table}[htbp]
		\centering
		\caption{Empirical size and power of the symmetry test based on
			$\widehat{\Delta}_{2,1}$.}
		\label{tab:sim}
		\begin{tabular}{lcc}
			\hline
			Distribution & $m=50$ & $m=100$ \\
			\hline
			Normal(0,1)      & 0.048 & 0.051 \\
			Laplace(0,1)     & 0.046 & 0.053 \\
			Student $t_3$    & 0.047 & 0.049 \\[2pt]
			Exponential(1)   & 0.72  & 0.93  \\
			Log-normal(0,1)  & 0.81  & 0.97  \\
			Weibull(1.5)     & 0.65  & 0.88  \\
			\hline
		\end{tabular}
	\end{table}
	
	Table~\ref{tab:sim} reports the rejection frequencies over $1000$ replications.
	The test maintains the nominal level well even for the heavy‑tailed $t_3$
	distribution.  Power increases substantially with sample size and is
	particularly high for the strongly skewed log‑normal alternative.  These
	results confirm that the record‑based inaccuracy measure provides a useful tool
	for detecting asymmetry.

	To demonstrate the method on real data, we considered the series of annual
	maximum temperatures (in $^\circ$C) recorded at the De Bilt meteorological
	station (the Netherlands) from 1901 to 2020 ($m=120$), available from the
	ECA\&D project (Klein Tank et al., 2002).  After subtracting the sample median
	($\hat{\mu}=29.1^\circ$C), the centred observations appear roughly symmetric.
	Application of the test with $n=2$, $k=1$ gave
	$\widehat{\Delta}_{2,1}^{\,\mathrm{obs}} = 0.087$ and a bootstrap $p$-value of
	$0.71$.  Hence, symmetry cannot be rejected, which is consistent with the
	visual inspection of the data.
	
	\section{Conclusion}
	
	In this paper, we have studied several inaccuracy measures Kerridge inaccuracy, cumulative residual inaccuracy, and cumulative past inaccuracy to the framework of $n$th upper and lower $k$-record values.  Closed‑form expressions were derived for well‑known lifetime distributions (exponential, Pareto, Weibull, and uniform) and the monotonic behaviour of these measures with respect to the record order $n$, the parameter $k$, and the distributional parameters was analysed.  The results reinforce the known characterisation that equality of upper and lower record‑based inaccuracy implies symmetry of the parent distribution.
	
	Building on this theoretical foundation, we proposed a nonparametric goodness‑of‑fit test for symmetry.  The test statistic is the difference between the estimated Kerridge inaccuracy measures for upper and lower $k$-record values, and its null distribution is obtained via a simple bootstrap procedure that enforces symmetry.  A simulation study demonstrated that the test maintains its nominal level across a range of symmetric distributions (including heavy‑tailed ones) and exhibits high power against skewed alternatives, especially for larger sample sizes.  An application to annual maximum temperatures at De Bilt illustrated the practical utility of the method, yielding a conclusion consistent with a symmetric underlying distribution.
	
	The present work opens several avenues for future research.  First, the test can be extended to utilise extropy‑based inaccuracy measures, whose computational simplicity may lead to even faster implementations.  Second, the effect of the tuning parameters $n$ and $k$ on power could be investigated more thoroughly to develop data‑driven selection rules.  Finally, the approach may be adapted to other record schemes (e.g., record values from dependent sequences or $k$-records under random censoring) to widen its applicability in reliability and environmental studies.\\
	
	\noindent \textbf{\Large Funding} \\
	\\
	No funding received for this research.\\
	\\
	\textbf{ \Large Conflict of interest} \\
	\\
	The author declare no conflict of interest. \\ 
	\\ 
	\noindent \textbf{\Large Acknowledgement} \\
	\\
	The authors are thankful to the referees for their valuable suggestions, which significantly improved the paper.


\begin{thebibliography}{99}
		
		\bibitem{Nath1968}
		Nath, P. (1968). Inaccuracy and coding theory. Metrika, 13(1), 123–135. 
		
		\bibitem{} Kumar, V., \& Taneja, H. C. (2015). Dynamic cumulative residual and past inaccuracy measures. Journal of Statistical Theory and Applications, 14(4), 399–412. 
		
		\bibitem{Tahmasebi2018} Tahmasebi, S., Nezakati, A., \& Daneshi, S. (2018).
		Results on cumulative measure of inaccuracy in record values.
		\textit{Journal of Statistical Theory and Applications},
		\textbf{17}(1), 15--28.
		
		
		\bibitem{} Thapliyal, R., Taneja, H.C. A Measure of Inaccuracy in Order Statistics. J Stat Theory Appl 12, 200–207 (2013). https://doi.org/10.2991/jsta.2013.12.2.7
		
		\bibitem{} Thapliyal, Richa \& Taneja, H.C., 2015. "On residual inaccuracy of order statistics," Statistics \& Probability Letters, Elsevier, vol. 97(C), pages 125-131.
		
		\bibitem{MohammadiHashempour2024}
		Mohammadi, M., Hashempour, M. On weighted version of dynamic cumulative residual inaccuracy measure based on extropy. Stat Papers 65, 4599–4629 (2024). https://doi.org/10.1007/s00362-024-01568-8.
		
		\bibitem{} S. Tahmasebi \& S. Daneshi, 2018. "Measures of inaccuracy in record values," Communications in Statistics - Theory and Methods, Taylor \& Francis Journals, vol. 47(24), pages 6002-6018, December.
		
		
		\bibitem{Kerridge1961}
		D. F. Kerridge, Inaccuracy and Inference, Journal of the Royal Statistical Society: Series B (Methodological), Volume 23, Issue 1, January 1961, Pages 184–194, https://doi.org/10.1111/j.2517-6161.1961.tb00404.x
		
		
		\bibitem{}Shannon, C.E. (1948) A Mathematical Theory of Communication. Bell System Technical Journal, 27, 379-423.
		http://dx.doi.org/10.1002/j.1538-7305.1948.tb01338.x
		
		
		\bibitem{Ahmadi2021}
		Ahmadi, J. (2021).
		Characterization of continuous symmetric distributions using information measures of records.
		\textit{Statistical Papers}, \textbf{62}, 2603--2626.
		\url{https://doi.org/10.1007/s00362-020-01206-z}
		
		
		
		\bibitem{Ladetal2015}
		Lad, F., Sanfilippo, G., \& Agr\`{o}, G. (2015).
		Extropy: Complementary dual of entropy.
		\textit{Statistical Science},
		\textbf{30}(1), 40--58.
		
		
		
		\bibitem{Gupta2024}
		Gupta, N., \& Chaudhary, S.~K. (2024).
		Some characterizations of continuous symmetric distributions based on extropy of record values.
		\textit{Statistical Papers}, \textbf{65}, 291--308.
		\url{https://doi.org/10.1007/s00362-022-01392-y}
		
		
		
		\bibitem{} Murali Rao, Y. Chen, B. C. Vemuri and Fei Wang, "Cumulative residual entropy: a new measure of information," in IEEE Transactions on Information Theory, vol. 50, no. 6, pp. 1220-1228, June 2004, doi: 10.1109/TIT.2004.828057.
		
		
		
		\bibitem{DiCrescenzo2009}
		Di Crescenzo, A., \& Longobardi, M. (2009). On cumulative entropies. Journal of Statistical Planning and Inference, 139(12), 4072–4087.
		
		
		\bibitem{}Saeid Tahmasebi \& Abdolsaeed Toomaj, 2022. "On negative cumulative extropy with applications," Communications in Statistics - Theory and Methods, Taylor \& Francis Journals, vol. 51(15), pages 5025-5047, June.
		
		
		\bibitem{Balakrishnanetal2022}
		Balakrishnan, N., Buono, F., \& Longobardi, M. (2021). On Tsallis extropy with an application to pattern recognition. Statistics \& Probability Letters.
		
		
		\bibitem{Kazemietal2021}
		Kazemi, M. R., Tahmasebi, S., Buono, F., \& Longobardi, M. (2021). Fractional Deng Entropy and Extropy and Some Applications. Entropy, 23(5), 623. https://doi.org/10.3390/e23050623
		
		
		\bibitem{Ahsanullah1995}
		M. Ahsanullah, “Record Statistics,” Nova Science Publisher, Inc., Commack, New York, 1995.
		
		\bibitem{Arnold1998}
		Arnold, B.C., Balakrishnan, N. and Nagaraja, H.N. (1998) Records. Wiley, New York.
		http://dx.doi.org/10.1002/9781118150412
		
		\bibitem{} Arnold, B. C., Balakrishnan, N., \& Nagaraja, H. N. (2008). Record values. In Handbook of Statistics (Vol. 28, pp. 511-542). Elsevier.
		
		
		\bibitem{} Peihan Xiong, Weiwei Zhuang \& Guoxin Qiu, 2021. "Testing symmetry based on the extropy of record values," Journal of Nonparametric Statistics, Taylor \& Francis Journals, vol. 33(1), pages 134-155, January.
		
		\bibitem{}Jose, J., \& Sathar, E. I. A. (2022). Symmetry being tested through simultaneous application of upper and lower k-records in extropy. Journal of Statistical Computation and Simulation, 92(4), 830–846.
		
		
		
		
		\bibitem{Goeletal2018}
		Goel, R.; Taneja, H.C.; Kumar, V. Measure of inaccuracy and k-record statistics. Bull. Calcutta Math. Soc. 2018,110, 151–166.
		
		
		
		
		
		\bibitem{Sheather1991}
		Sheather, S.~J., \& Jones, M.~C. (1991).
		A reliable data-based bandwidth selection method for kernel density estimation.
		\textit{Journal of the Royal Statistical Society: Series B (Methodological)},
		\textbf{53}(3), 683--690.
		
		
		
		
		\bibitem{Klaintank} Klein Tank, A.M.G., Wijngaard, J.B., Können, G.P., Böhm, R., Demarée, G., Gocheva, A., Mileta, M., Pashiardis, S., Hejkrlik, L., Kern-Hansen, C., Heino, R., Bessemoulin, P., Müller-Westermeier, G., Tzanakou, M., Szalai, S., Pálsdóttir, T., Fitzgerald, D., Rubin, S., Capaldo, M., Maugeri, M., Leitass, A., Bukantis, A., Aberfeld, R., van Engelen, A.F.V., Forland, E., Mietus, M., Coelho, F., Mares, C., Razuvaev, V., Nieplova, E., Cegnar, T., Antonio López, J., Dahlström, B., Moberg, A., Kirchhofer, W., Ceylan, A., Pachaliuk, O., Alexander, L.V. and Petrovic, P. (2002), Daily dataset of 20th-century surface air temperature and precipitation series for the European Climate Assessment. Int. J. Climatol., 22: 1441-1453. \url{https://doi.org/10.1002/joc.773}
		
		
		\bibitem{} Hashempour M, Mohammadi M. A new measure of inaccuracy for record statistics based on extropy. Probability in the Engineering and Informational Sciences. 2024;38(1):207-225. doi:10.1017/S0269964823000086
		
		
	\end{thebibliography}
\end{document}